\documentclass[final, a4paper, 12pt]{article}

\usepackage[scale=0.768]{geometry}
\usepackage[english]{babel}

\usepackage[mathscr]{euscript}
\usepackage{color}
\usepackage{fancyhdr}
\usepackage{dsfont}
\usepackage{bm}
\usepackage{hyperref}
\usepackage{cite}
\usepackage{showkeys}
\usepackage{subfigure}
\usepackage{float}

\usepackage{amsfonts}
\usepackage{amsmath}
\usepackage{amssymb}
\usepackage{amscd}
\usepackage{amsthm}

\usepackage{dblaccnt}
\usepackage{accents}
\usepackage{graphicx}
\usepackage{array}

\newtheorem{theorem}{Theorem}

\newtheorem{lemma}[theorem]{Lemma}

\newtheorem{remark}[theorem]{Remark}
\newtheorem{assumption}[theorem]{Assumption}

\newcommand*{\N}{\ensuremath{\mathbb{N}}}
\newcommand*{\Z}{\ensuremath{\mathbb{Z}}}
\newcommand*{\R}{\ensuremath{\mathbb{R}}}
\newcommand*{\C}{\ensuremath{\mathbb{C}}}
\renewcommand{\i}{\mathrm{i}}
\renewcommand{\phi}{\varphi}
\renewcommand{\rho}{{\varrho}}
\renewcommand{\epsilon}{{\varepsilon}}

\renewcommand{\d}[1]{\,\mathrm{d}#1 \,}

\newcommand{\J}{\mathcal{J}} 
\newcommand{\0}{{0}} 

\newcommand{\grad}{\nabla}

\newcommand{\W}{{W_{\hspace*{-1pt}{\Lambda}}}} 
\newcommand{\Wast}{{W_{\hspace*{-1pt}{\Lambda}^\ast}}} 

\newlength{\dhatheight}

\usepackage{color}

\begin{document}

\sloppy

\title{A Bloch transform based  numerical method for the rough surface scattering problems}
\author{Ruming Zhang\thanks{Center for Industrial Mathematics, University of Bremen
; \texttt{rzhang@uni-bremen.de}}}
\date{}
\maketitle

\begin{abstract}
In this paper, we will study the Bloch transformed  rough surface scattering problems, and propose a numerical method based on the Bloch transformed problems. Based on the mathematical theory of the scattering problems from locally perturbed periodic surfaces, the same techniques will be applied to the rough surface scattering problems, and an equivalent coupled family of quasi-periodic scattering problems in one periodic cell will be established. The most important result obtained in this paper is on the finite Fourier series approximation of the Bloch transformed field with respect to the quasi-periodicity parameter. It will be proved that the finite series is exactly the Bloch transformed solution corresponds to truncated rough surfaces. Thus the truncation provides a reasonable approximation, and could be applied to the numerical solutions. Based on the approximation, a numerical method is proposed for the rough surface scattering problems. The convergence of the numerical method is proved and illustrated by the numerical experiments. The method provides a completely new perspective for the rough surface scattering problems. There is possibility that some high order method will be developed based on this new method.
\end{abstract}

\section{Introduction}

The scattering problems from rough structures are always interesting but challenging topics. Mathematicians have been working on both the theoretical analysis and numerical implementation of this topic for decades. Based on integral equations, the well-posedness of the problems has been investigated, see \cite{Chand1996,Chand1998,Chand1999,Zhang2003}. A Nystr\"{o}m method for the integral equation on the real line has been developed for the rough surface scattering problems, see \cite{Meier2000,Arens2002}. In 2005, Chandler-Wilde and Monk proposed a  variational method (see \cite{Chand2005}) for the  investigation of the well-posedness of the scattering problems from rough surfaces in both 2D and 3D spaces. Based on the variational method, the unique solvability of the scattering from rough surfaces has been proved in weighted Sobolev spaces in \cite{Chand2010}, in which more generalized cases (e.g. plane waves in 2D spaces) are included. Similar results in weighted Sobolev spaces has been shown for more generalized boundary conditions in \cite{Hu2015}. Based on the variational formulation in weighted spaces, a so-called "finite section method" has been proposed for numerical approximations in \cite{Chand2010}.

Recently, a Floquet-Bloch transform based method has been proposed to treat  scattering problems from a special kind of rough surfaces, i.e., the locally perturbed periodic surfaces. To the author's knowledge, this method was first introduced in the paper \cite{Coatl2012}, for scattering problems from locally perturbed periodic mediums. For the non-periodic incident fields  (e.g., the Herglotz wave functions or the point sources) scattered by locally perturbed periodic surfaces, the scattering problems are transformed into an equivalent coupled family of quasi-periodic scattering problems by the Bloch transform (see \cite{Lechl2015e,Lechl2016}). Based on the theoretical analysis for the Bloch-transformed scattering problems, convergent numerical methods have been developed for the scattering problems, for periodic surfaces see \cite{Lechl2016a}, and for  locally perturbed periodic surfaces see \cite{Lechl2017}. Discussions are also made for 3D acoustic and electromagnetic problems, and the first numerical experiment has been carried out for the 3D Helmholtz equations, see \cite{Lechl2016b}. For the scattering problems from locally perturbed periodic  layers, similar method has been adopted in \cite{Hadda2015}.  The Bloch transform was also applied to the scattering problems in locally perturbed periodic waveguides, see \cite{Fliss2015}.

In this paper, the Floquet-Bloch transform will be applied to the rough surface scattering problems. With the technique introduced in \cite{Lechl2016}, the problem is transformed into one defined in an infinite rectangle. With the help of the Bloch transform, the new problem is decomposed into a coupled family of quasi-periodic scattering problems, and the equivalence between the original problem and the coupled family of quasi-periodic problems will be proved. The main difficulty comes from the term involves the perturbation, which is no longer compactly supported, as that for  locally perturbed cases. An interesting result shows that, the approximated finite Fourier series of the Bloch transformed field with respect to the quasi-periodicity parameter is exactly the Bloch transformed solution of the scattering problem with truncated rough surface. It could be a good choice of approximation for the scattering problems in unbounded domains, for the numerical method of the locally perturbed problems have been well studied in \cite{Lechl2017}, and a high order method was developed in \cite{Zhang2017e}. Based on this result, a finite element method will be developed for the numerical solution. Although there are convergent numerical methods for the simulation of the scattering problems, e.g., the one based on the integral equations (see \cite{Arens2002,Meier2000}) or variational formulations (see \cite{Chand2010}), the new method may provide very different perspective for the rough surface scattering problems. It is also expected that the Bloch transform based numerical method will be improved and a high order method will be developed following \cite{Zhang2017e} in the future.

The rest of the paper is organized as follows. In Section 2,  we will recall the mathematical model of the rough surface scattering problems and its unique solvability. In Section 3, we will formulate the weak formulation of the Bloch transformed scattering problems, and study the equivalence and unique solvability of the newly established variational problem. In Section 4, the finite Fourier series  approximation of the Bloch transformed fields will be studied. Based on the approximation, a finite element method will be proposed in Section 5, and details of the numerical implementation will be explained in Section 6. In the last section, we will give several numerical experiments to illustrate the convergence result obtained in Section 5

\section{Scattering from rough surfaces}

In this section, we  recall the mathematical modal and well-posedness of the scattering problems from rough surfaces in two dimensional spaces. For details we refer to  the papers \cite{Chand2005,Chand2010}.

Let $\zeta$ be a bounded function defined in $\R$.  
In this paper, we have to require that the following assumption holds for the functions $\zeta$. 
\begin{assumption}\label{asp1} 
$\zeta$ is a Lipschitz continuous function defined on $\R$. Suppose there is a positive $C>0$ such that
\begin{equation*}
\left\|\zeta\right\|_{1,\infty}\leq C.
\end{equation*}
Without loss of generality, assume that $\zeta(t)\geq c>0$ for any $t\in\R$.
\end{assumption}

\begin{figure}[H]
\centering
\includegraphics[width=15cm]{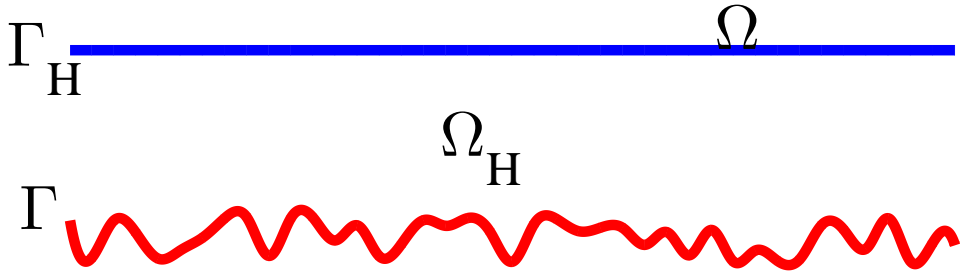}
\caption{A global rough surface.}
\end{figure}

Define the rough surface by
\begin{equation*}
\Gamma:=\left\{(x_1,\zeta(x_1)):\,x_1\in\R\right\}.
\end{equation*}
Let $\Gamma_h$ be defined by the straight line $\R\times \{h\}$ for any $h\in\R$. Suppose $H$ is a constant satisfies $H>\max\left\{\|\zeta\|_\infty\right\}$, then $\Gamma_H$ is a straight line lies above $\Gamma$. Define the domains  by
\begin{equation*}
\Omega:=\left\{(x_1,x_2):\,x_2>\zeta(x_1)\right\},\quad\Omega_H:=\left\{(x_1,x_2):\,\zeta(x_1)<x_2<H\right\}.
\end{equation*}

Given an incident field $u^i$ that satisfies the Helmholtz equation 
\begin{equation*}
\Delta u^i+k^2 u^i=0 \text{ in } \Omega,
\end{equation*}
it is scattered by the impenetrable surface $\Gamma$. Then the total field $u$  satisfies the Helmholtz equation in $\Omega$ as well, i.e.,
\begin{equation}
\Delta u+k^2 u=0\text{  in }\Omega.
\end{equation}
Assume that the total field $u$ also satisfies the homogeneous Dirichlet boundary condition on the  boundary $\Gamma$, i.e.,
\begin{equation}
 u=0\text{ on }\Gamma.
\end{equation}

\begin{remark}
In this paper, only the Dirichlet boundary condition is considered. However, it is possible to treat problems with different conditions, such as the impedance boundary condition (see \cite{Lechl2016}) or inhomogeneous mediums. 
\end{remark}

The scattered field $u^s:=u-u^i$ satisfies the so-called angular spectrum representation as a radiation condition (see \cite{Chand2005}), i.e.,
\begin{equation}\label{eq:ASR}
u^s(x_1,x_2)=\frac{1}{2\pi}\int_{\R} e^{\i x_1\cdot{\xi}+\i\sqrt{k^2-|{\xi}|^2}(x_2-H)}\widehat{u^s}({\xi},H)\d{\xi},\quad x_2\geq H,
\end{equation}
where $\widehat{u}^s$ is the Fourier transform of $u^s$ on $\Gamma_H$, and $\sqrt{k^2-|{\xi}|^2}=\i\sqrt{|\xi|^2-k^2}$ when $|\xi|>k$. The radiation condition \eqref{eq:ASR} is equivalent to the following boundary condition 
\begin{equation*}
\frac{\partial u^s}{\partial x_2}(x_1,H)=T^+\left[u^s\big|_{\Gamma_H}\right],\quad\text{ on } \Gamma_H,
\end{equation*}
where $T^+$ is the Dirichlet-to-Neumann map defined by
\begin{equation}\label{eq:DtN}
T^+\phi=\frac{\i}{{2\pi}}\int_{\R} \sqrt{k^2-|{\xi}|^s}e^{\i x_1\cdot{\xi}}\widehat{\phi}({\xi})\d{\xi}\quad\text{ for }\phi=\frac{1}{2\pi}\int_{\R} e^{\i x_1\cdot{\xi}}\widehat{\phi}({\xi})\d{\xi}.
\end{equation}
$T^+$ is a continuous operator from $H_r^{1/2}(\Gamma_H)$ into $H_r^{-1/2}(\Gamma_H)$ for any $|r|<1$ (see \cite{Chand2010}).
Thus the total field satisfies the boundary condition on $\Gamma$
\begin{equation}\label{eq:boundary_condition}
\frac{\partial u}{\partial x_2}(x_1,H)=T^+\left[u\big|_{\Gamma_H}\right]+f,\quad \text{ where }f:=\frac{\partial u^i}{\partial x_2}(x_1,H)-T^+\left[u^i|_{\Gamma_H}\right].
\end{equation}
The scattering problem is now turned into a problem that defined on  the domain $\Omega_H$ with a finite height. 
The weak formulation for the scattering problem is, given any $f\in H^{-1/2}(\Gamma_H)$, to find a solution $u\in\widetilde{H}^1(\Omega_H)$ such that
\begin{equation}\label{eq:var_origional}
\int_{\Omega_H}\left[\nabla u\cdot\nabla\overline{v}-k^2u\overline{v}\right]\d x-\int_{\Gamma_H}T^+\left[u|_{\Gamma_H}\right]\overline{v}\d s=\int_{\Gamma_H}f\overline{v}\d s,
\end{equation}
for all $v\in\widetilde{H}^1(\Omega_H)$ with compact support in $\overline{\Omega_H}$. The variational problem could also be analysed in the weighted Sobolev space $\widetilde{H}_r^1(\Omega_H)$.
\begin{remark}
The tilde in $\widetilde{H}_r^1(\Omega_H)$ shows that the functions in this space belong to $H_r^1(\Omega_H)$ and satisfy homogeneous Dirichlet boundary condition on $\Gamma$. Similar notations are utilized for other spaces, e.g., $H_0^r(\Wast;\widetilde{H}_{\alpha}^s(D^\Lambda_H))$.
\end{remark}

From \cite{Chand2010}, the unique solubility of the variational problem \ref{eq:var_origional} has been proved in weighted Sobolev spaces.
\begin{theorem}\label{th:solv}
If $\Gamma$ is Lipschitz continuous, $f\in H_r^{-1/2}(\Gamma_H)$ for $|r|<1$, then there is a unique solution $u\in\widetilde{H}_r^1(\Omega_H)$ for the variational problem \eqref{eq:var_origional}.
\end{theorem}

\section{The Bloch transform and the scattering problems}

In this section, we  apply the Bloch transform to  the scattering problems. As the Bloch transform only works on functions defined in periodic domains, the first step is to transform the problem into one defined in a periodic domain, and then apply the Bloch transform to the new problem (see  \cite{Lechl2016,Lechl2016b}). In this paper,  the procedure will be altered slightly, i.e., to transform the scattering problem into the new one defined in the infinite rectangle $D_H:=\R\times[h_0,H]$ where  $0<h_0,\,\|\zeta\|_\infty<H$.  Let $\Gamma^\Lambda_H$ and $D^\Lambda_H$ be  $\Gamma_H$ and $D_H$ restricted in one periodic cell $\W\times\R$, i.e.,
\begin{equation*}
\Gamma^\Lambda_H=\Gamma_H\cap\left[\W\times\R\right],\quad D^\Lambda_H=D_H\cap\left[\W\times\R\right].
\end{equation*} 

\begin{figure}[H]
\centering
\includegraphics[width=15cm]{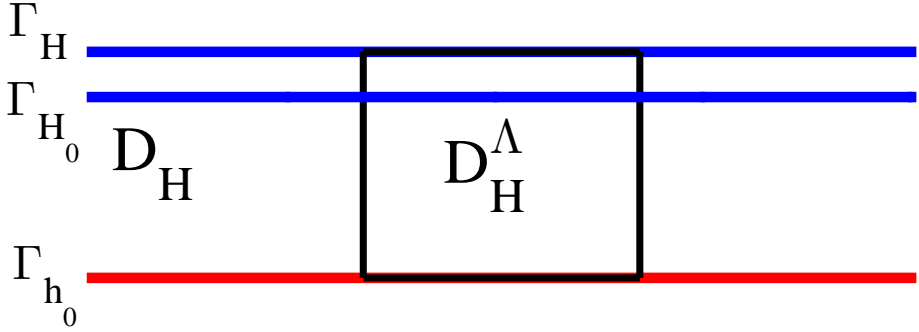}
\caption{Strip $D_H$ and one periodic cell $D^\Lambda_H$.}
\end{figure}

 Let $\Theta$ be a diffeomorphism that maps $\Omega_{H_0}$ to $D_{H_0}$ for some $\|\zeta\|_\infty<H_0<H$, and extend $\Theta$ by identity in $\R\times[H_0,\infty)$. Thus the support of $\Theta-I$ is contained in $D_{H_0}$.
\begin{remark}\label{asp2}
With Assumption \ref{asp1}, we can always find a diffeomorphism $\Theta$, such that $\Theta$ and $\Theta^{-1}$ are both Lipschitz continuous in $\Omega_H$. An example for the definition of $\Theta$ is
\begin{equation}\label{eq:Phi_p}
\Theta:\,{ x}\mapsto\left(x_1,x_2+\frac{(H_0-x_2)^3}{(H_0-h_0)^3}\left(\zeta(x_1)-h_0\right)\right)\quad\text{ when }x\in D_{H_0},
\end{equation}
and then extend it by the identity operator $I$ when $x_2\geq H_0$.
\end{remark}

 Let the transformed total field $u_T:=u\,\circ\,\Theta$,  then by direct calculation, $u_T\in\widetilde{H}_r^1(D_H)$   satisfies the following variational problem in the periodic domain $D_H$
\begin{equation}\label{eq:var_T}
\int_{D_H}\left[A\nabla u_T\cdot\nabla\overline{v_T}-k^2c u_T\overline{v_T}\right]\d { x}-\int_{\Gamma_H}T^+\left[u_T|_{\Gamma_H}\right]\overline{v_T}\d s=\int_{\Gamma_H}f\overline{v_T}\d s,
\end{equation}
for all $v_T:=v\circ \Theta\in\widetilde{H}^1(D_H)$, 
\begin{eqnarray*}
&& A_\Theta({x}):=\left|\det\grad\Theta({ x})\right|\left[\left(\grad\Theta({ x})\right)^{-1}\left(\grad\Theta({ x})\right)^{-T}\right]\in L^\infty\left(\Omega_H,\R^{2\times 2}\right),\\
&& c_\Theta({ x}):=\left|\det\grad\Theta({ x})\right|\in L^\infty (\Omega_H).
\end{eqnarray*}
Thus the support of both $A_\Theta$ and $c_\Theta$ are subsets of $D_{H_0}$.
\begin{remark}
From the definitions of $A_\Theta$ and $c_\Theta$, if $\zeta$ has higher regularities, e.g., if $\zeta\in C^{n,1}(\R)$ for some positive integer $n\geq 1$, $A_\Theta$ and $c_\Theta$ will have higher regularities as well, i.e., $A_\Theta\in W^{n-1,\infty}(\Omega_H,\R^{2\times 2})$ and $c_\Theta\in W^{n-1,\infty}(\Omega_H)$. Moreover, there is a constant $C>0$ depends only on $\zeta$ such that 
\begin{equation}\label{eq:dependence}
\|A_\Theta\|_{ W^{n-1,\infty}(\Omega_H,\R^{2\times 2})},\,\|c_\Theta\|_{W^{n-1,\infty}(\Omega_H)}\leq C\|\zeta\|_{W^{n,\infty}(\R)}
\end{equation}
\end{remark}

From direct calculations, we rewrite the variational problem \eqref{eq:var_T} as
\begin{equation*}
\begin{aligned}
&\int_{D_H}\left[\nabla u_T\cdot\nabla\overline{v_T}-k^2u_T\overline{v_T}\right]\d { x}-\int_{\Gamma_H}T^+\left[u_T|_{\Gamma_H}\right]\overline{v_T}\d s\\
&\qquad+\int_{D_H}\left[(A_\Theta-I)\nabla u_T\cdot\nabla\overline{v_T}-k^2(c_\Theta-1)u_T\overline{v_T}\right]\d { x}
=\int_{\Gamma_H}f\overline{v_T}\d s
\end{aligned}.
\end{equation*}

Use the property of the Bloch transform, let $w_B=\J_{D_H} u_T$, $v_B=\overline{\J_{D_H} \overline{v}}$, then $w_B\in H_0^r(\Wast;\widetilde{H}_\alpha^1(D^\Lambda_H))$, the variational form is equivalent to
\begin{equation*}
\begin{aligned}
\int_\Wast a_{\alpha}(w_B({\alpha},\cdot),v_B({\alpha},\cdot))\d{\alpha}+\int_{D_H} \left[(A_\Theta-I)\grad u_T\cdot\grad\overline{v}-k^2 (c_\Theta-1)u_T\overline{v}\right]\d{ x}\\
=\int_\Wast\int_{\Gamma^\Lambda_H}\left[\J_{\Gamma_H} f\right](\alpha,x) \overline{v(\alpha,x)}\d{ s(x)}\d\alpha,
\end{aligned},
\end{equation*}
for any $v_B\in H_0^{-r}(\Wast;H^1_{\alpha}(D^\Lambda_H))$ with compact support, where
\begin{equation*}
 a_{\alpha}(w,v):=\int_{D^\Lambda_H}\left[\nabla w\cdot\nabla \overline{v}-k^2 w\overline{v}\right]\d{ x}-\int_{\Gamma^\Lambda_H}T^+_{\alpha}(w)\overline{v}\d s \text{ for } w,v\in H^1_\alpha(D^\Lambda_H)
\end{equation*}
and $T^+_{\alpha}$ is the ${\alpha}$-quasi-periodic Dirichlet-to-Neumann operator defined by
\begin{equation*}
T_{\alpha}^+\phi=\i\sum_{{ j}\in\Z}\sqrt{k^2-|\Lambda^* j-{\alpha}|^2}\widehat{\phi}({ j})e^{\i(\Lambda^*{ j}-{\alpha})\cdot x_1}\quad\text{ for }\phi=\sum_{{ j}\in\Z}\widehat{\phi}({ j})e^{\i(\Lambda^*{ j}-{\alpha})\cdot x_1}.
\end{equation*}

Define $c:=c_\Theta-1$ and $A:=A_\Theta-I$, then both of them are bounded due to the boundedness of $\zeta$ in $W^{1,\infty}(\R)$ (see \eqref{eq:dependence}). As $u_T\in H_r^1(D_H)$ and $v=\overline{\J^{-1}_{D_H} \overline{v_B}}\in H_{-r}^1(D_H)$, the integrals $\int_{D_H}A\grad u_T\cdot\grad\overline{v}\d{ x}$ and $\int_{D_H}c u_T\overline{v}\d { x}$ are well defined and bounded by
\begin{equation*}
\left|\int_{D_H}A\grad u_T\cdot\grad\overline{v}\d{ x}\right|,\, \left|\int_{D_H}c u_T\overline{v}\d { x}\right|\leq C\|\zeta\|_{1,\infty}\left\|u_T\right\|_{H_r^1(D_H)}\|v\|_{H_{-r}^1(D_H)}.
\end{equation*}
For the mapping property of the Bloch transform (see Appendix),
\begin{equation*}
\left|\int_{D_H}A\grad u_T\cdot\grad\overline{v}\d{ x}\right|,\, \left|\int_{D_H}c u_T\overline{v}\d { x}\right|\leq  C\|\zeta\|_{1,\infty}\left\|w_B\right\|_{H_0^r(\Wast;H^1_\alpha(D_H))}\|v_B\|_{H_0^{-r}(\Wast;H^1_\alpha(D_H))}.
\end{equation*}
Define the sesquilinear form $b(\cdot,\cdot)$ 
\begin{equation*}
b(w,v)=\int_{D_H}\Big[ A\grad( \J_{D_H}^{-1} w)\cdot\grad\overline{\left(\J_{D_H}^{-1} \overline{v}\right)}-k^2 c(\J_{D_H}^{-1}w)\cdot\overline{\left(\J_{D_H}^{-1} \overline{v}\right)}\Big]\d x,
\end{equation*}
then it is a bounded  on $H_0^r(\Wast;\widetilde{H}_\alpha^1(D^\Lambda_H))\times H_0^{-r}(\Wast;\widetilde{H}_\alpha^1(D^\Lambda_H))$.

Finally we arrive at the variational formulation, i.e., for any $v_B\in L^2(\Wast;\widetilde{H}_\alpha^1(D^\Lambda_H))$, $w_B$ satisfies
\begin{equation}\label{eq:var_Bloch}
\int_\Wast a_\alpha(w_B(\alpha,\cdot),v_B(\alpha,\cdot))\d\alpha+\, b(w_B,v_B)=\int_\Wast\int_{\Gamma^\Lambda_H}F(\alpha,x) \overline{v(\alpha,x)}\d{ s(x)}\d\alpha,
\end{equation}
where 
\begin{equation*}
F(\alpha,x)=\left(\J_{\Gamma_H} f\right)(\alpha,x)\in H_0^r(\Wast;H_\alpha^{-1/2}(\Gamma^\Lambda_H)).
\end{equation*}

From the arguments above,  the equivalence between the weak formulation \eqref{eq:var_origional} of the scattering problem and the variational problem \eqref{eq:var_Bloch} is concluded in the following lemma.

\begin{lemma}\label{th:equivalent}
Assume that $f\in H_r^{-1/2}(\Gamma_H)$ for some $r\in[0,1)$, then $u\in \widetilde{H}_r^1(D_H)$ satisfies \eqref{eq:var_origional} if and only if $w_B=\J_{D_H} u_T\in H_0^r(\Wast;\widetilde{H}_\alpha^1(D^\Lambda_H))$  satisfies \eqref{eq:var_Bloch} for $F\in H_0^r(\Wast;H_\alpha^{-1/2}(\Gamma^\Lambda_H))$.
\end{lemma}

With the equivalence between \eqref{eq:var_origional} and \eqref{eq:var_Bloch} in Lemma \ref{th:equivalent}, we will show the unique solvability of the variational problem \eqref{eq:var_Bloch}.

\begin{theorem}
Suppose $\zeta$ is Lipschitz continuous. Given any $F\in H_0^r(\Wast;\widetilde{H}_\alpha^{-1/2}(\Gamma_H^\Lambda))$ for some $r\in[0,1)$,  the variational problem \eqref{eq:var_Bloch} has a unique solution in $H_0^r(\Wast;\widetilde{H}^1_\alpha(D^\Lambda_H))$.
\end{theorem}

\begin{proof}

The first step is to prove the existence of the solution of \eqref{eq:var_Bloch}. Given a function $F\in H_0^r(\Wast;{H}_\alpha^{-1/2}(\Gamma_H^\Lambda))$, then $f:=\J_{\Gamma_H}^{-1}F\in H_r^{-1/2}(\Gamma_H)$. As $\Gamma$ is Lipschitz continuous, from Theorem \ref{th:solv}, there is a unique solution $u\in \widetilde{H}^1_r(\Omega_H)$ to the problem \eqref{eq:var_origional}.  From Lemma \ref{th:equivalent}, $w_B=\J_{D_H} u_T\in H_0^r(\Wast;\widetilde{H}^1_\alpha(D^\Lambda_H))$ is a solution to the variational problem \eqref{eq:var_Bloch}. 

Then let's prove the uniqueness of the solution. 
Suppose $w_B\in H_0^r(\Wast;\widetilde{H}^1_\alpha(D^\Lambda_H))$ is a solution to the problem \eqref{eq:var_Bloch} with  $F=0$, then $u_T:=\J_{D_H}^{-1}w_B$ is a solution to the problem \eqref{eq:var_origional} with $f=0$.  From the unique solvability of $\eqref{eq:var_origional}$, $u=0$, thus $w_B=0$.  The proof is finished.
\end{proof}

Theorem 7 and Theorem 8 in \cite{Lechl2016b} could be extended to the rough surface scattering problems. With the assumptions the incident fields or the surfaces have higher regularities, the Bloch transformed fields are also smoother.  

\begin{theorem}\label{th:higher_reg}
Assume that $f\in H_r^{-1/2}(\Gamma_H)$ for some $r\in[0,1)$, and $\zeta\in C^{2,1}(\R)$. Then the solution $w_B\in H_0^r(\Wast;\widetilde{H}^2_\alpha(D^\Lambda_H))$ and $u_T=\J_{D_H}^{-1} w_B\in\widetilde{H}^2_r(\Omega_H)$.
\end{theorem}

When $u^i$ decays fast enough at the infinity, i.e., $u^i\in H_r^1(D_H)$ for $r\in(1/2,1)$, the Bloch transform $\J_{D_H}u^i\in H_0^r(\Wast;H^1_\alpha(D^\Lambda_H))$ depends continuously on $\alpha$, then the Bloch transformed field $w_B\in H_0^r(\Wast;\widetilde{H}_\alpha^1(D^\Lambda))$. Then the following equivalent formulation holds.

\begin{theorem}
If $f\in H_r^{-1/2}(\Omega_H)$ for some $r\in(1/2,1)$, then the solution $w_B\in H_0^r(\Wast;\widetilde{H}^1_\alpha(D^\Lambda_H))$ equivalently satisfies for all $\alpha\in\Wast$ and $v_\alpha\in\widetilde{H}_\alpha^1(D^\Lambda_H)$ such that
\begin{equation}\label{eq:var_Bloch_continuous}
a_\alpha(w_B(\alpha,\cdot),v_\alpha)+ b(w_B(\alpha,\cdot),v_\alpha)=\int_{\Gamma^\Lambda_H}F(\alpha,x)\overline{v_\alpha(x)}\d s(x).
\end{equation}

\end{theorem}

In this section, the variational formulation for the Bloch transformed total field has been established with the help of the properties of the Bloch transform. Similar to the special case, i.e., locally perturbed periodic surfaces (see \cite{Lechl2016}), the variational problem is proved to be equivalent to the original problems, and is also uniquely solvable in certain Sobolev spaces.

\begin{remark}
For the special case that the surface is a small perturbation of a periodic one, the same technique in \cite{Lechl2016} could be adopted, i.e., to transform the problem into one defined in a periodic domain. Then the Bloch transformed field is analysed in one periodic cell of the periodic domain. It might be more convenient to solve the problems numerically in this way for the special case, with the same method as the generalized cases.
\end{remark}

\section{Finite dimensional approximation of scattering from rough surfaces}

From the variational problem \eqref{eq:var_Bloch}, the main difference between the globally and locally perturbed problems is the term $b(w_B,v_B)$. When the perturbation is local, the term $b(\cdot,\cdot)$ is reduced into the integral in a bounded domain, thus it is easy to be discretized following the method introduced in \cite{Lechl2017}. While when the perturbation is global, the numerical algorithm becomes much more difficult due to the infinite domain. 

 In the first subsection, the simplified case on the real line will be studied. The extension to 2D strip $D_H$ will be presented in the second subsection. The investigation of the finite-dimensional approximated field will be carried out in the third subsection.

\subsection{The simplified case: one dimensional problems}

In this subsection, we will consider the following integral
\begin{equation*}
\int_\R \rho(x_1)\left(\J_{\R}^{-1}\phi\right)(x_1)\overline{\left(\J_{\R}^{-1}\overline{\psi}\right)}(x_1)\d x_1,\quad x_1\in \R,
\end{equation*}
where $\phi\in H^r_0(\Wast;L^2(\W))$ and $\psi\in H^{-r}_0(\Wast;L^2(\W))$, $r\in\R$, and the operator $\J_\R$ is the Bloch transform defined on the real line. From Remark \ref{rem:Four}, $\phi$ could be defined by Fourier series, i.e.,
\begin{eqnarray*}
 \phi(\alpha,x_1)=C_\Lambda\sum_{\ell\in\Z}\widehat{\phi}_{\Lambda^*}(\ell,x_1)e^{-\i\alpha\cdot\Lambda\ell},\quad\widehat{\phi}_{\Lambda^*}(\ell,\cdot)=\left<\phi,\phi^{(\ell)}_{\Lambda^*}\right>_{L^2(\Wast)}\in L^2(\W).
\end{eqnarray*}

Define the finite dimensional space by
\begin{equation*}
X_N\left(\Wast;L^2(\W)\right):=\left\{\phi(\alpha,x_1)=\sum_{\ell\in\Z_N}c_\ell(x_1) e^{-\i\alpha\cdot\Lambda\ell}:\, c_\ell\in L^2(\W)\right\}
\end{equation*}
where $\Z_N=\left\{-N/2+1,\dots,N/2\right\}$ when $N$ is even. 
\begin{remark}
For simplicity, we assume that $N$ are even numbers.
\end{remark}
The approximation of $\phi$ in  the subspace $X_N\left(\Wast;L^2(\W)\right)$ is given by
\begin{equation*}
\phi_N(\alpha,x_1)=\sum_{\ell\in\Z_N}\widehat{\phi}_{\Lambda^*}(\ell,x_1)e^{-\i\alpha\cdot\Lambda\ell}.
\end{equation*}
From direct computation, the inverse Bloch transforms of $\phi(\alpha,x_1)$ and $\phi_N(\alpha,x_1)$
\begin{eqnarray*}
&&\left(\J_\R^{-1}\phi\right)(x_1+\Lambda j)=\widehat{\phi}_{\Lambda^*}(j,x_1);\\
&&\left(\J_\R^{-1}\phi_N\right)(x_1+\Lambda j)=\widehat{\phi}_{\Lambda^*}(j,x_1)\delta_{j,\ell}
\end{eqnarray*}
in the weighted Sobolev space $H_r^0(\R)$, where $\delta_{j,\ell}$ equals to 1 when $j=\ell$ and equals to 0 otherwise. Thus
$\J_\R^{-1}\phi_N$ is compactly supported in $\cup_{j\in\Z_N}\left[\W+\Lambda j\right]$. Define the indicator function $\mathcal{X}$ by
\begin{equation*}
\mathcal{X}_N(t)=\begin{cases}
1,\quad\text{ in }\cup_{j\in\Z_N}\left[\W+\Lambda j\right];\\
0,\quad\text{otherwise},
\end{cases}
\end{equation*}
then
\begin{equation*}
\left(\J_\R^{-1}\phi_N\right)(x_1)=\left(\J_\R^{-1}\phi\right)(x_1)\mathcal{X}_N(x_1).
\end{equation*}
Thus the integral satisfies
\begin{equation*}
\int_\R\rho(x_1)\left(\J_{\R}^{-1}\phi_N\right)(x_1)\overline{\left(\J_{\R}^{-1}\overline{\psi}\right)}(x_1)\d x_1=\int_{\R}\rho(x_1)\mathcal{X}_N(x_1)\left(\J_{\R}^{-1}\phi\right)(x_1)\overline{\left(\J_{\R}^{-1}\overline{\psi}\right)}(x_1)\d x_1
\end{equation*}
Define the truncated function
\begin{equation*}
\rho_N(x_1):=\rho(x_1)\mathcal{X}_N(x_1),
\end{equation*}
then
\begin{equation*}
\int_\R\rho(x_1)\left(\J_{\R}^{-1}\phi_N\right)(x_1)\overline{\left(\J_{\R}^{-1}\overline{\psi}\right)}(x_1)\d x_1=\int_{\R}\rho_N(x_1)\left(\J_{\R}^{-1}\phi\right)(x_1)\overline{\left(\J_{\R}^{-1}\overline{\psi}\right)}(x_1)\d x_1
\end{equation*}
We can also define the finite Fourier series of $\psi_N$ from $\psi$ similarly, then the following relationship could be obtained
\begin{equation}\label{eq:approx_finite}
\begin{aligned}
\int_\R\rho(x_1)\left(\J_{\R}^{-1}\phi_N\right)(x_1)\overline{\left(\J_{\R}^{-1}\overline{\psi}\right)}(x_1)\d x_1&=\int_\R\rho(x_1)\left(\J_{\R}^{-1}\phi\right)(x_1)\overline{\left(\J_{\R}^{-1}\overline{\psi_N}\right)}(x_1)\d x_1\\&=\int_{\R}\rho_N(x_1)\left(\J_{\R}^{-1}\phi\right)(x_1)\overline{\left(\J_{\R}^{-1}\overline{\psi}\right)}(x_1)\d x_1.
\end{aligned}
\end{equation}

\subsection{Extension to the sesquilinear form $b(w,v)$}

Similar to the one-dimensional case, we can also approximate $w_B$ in a finite dimensional subspace with respect to $\alpha$. Let the subspace of $H_0^r(\Wast;\widetilde{H}^1_\alpha(D^\Lambda_H))$ by
\begin{equation*}
X_N\left(\Wast;\widetilde{H}^1_\alpha(D^\Lambda_H)\right):=\left\{\phi(\alpha,x)=\sum_{\ell\in\Z_N}c_\ell(x) e^{-\i\alpha\cdot\Lambda\ell}:\, c_\ell\in \widetilde{H}^1_\alpha(D^\Lambda_H)\right\},
\end{equation*}

From the definition of $w_B=\J_\Omega u_T$, it has the representation
\begin{equation*}
w_B(\alpha,x)=C_\Lambda\sum_{j\in\Z}u_T\left(x+\left(\begin{matrix}
\Lambda j\\0
\end{matrix}
\right)\right)e^{-\i\alpha\Lambda j}.
\end{equation*}
Let $N>0$ be any even positive integer, then the approximation of $w_B$ in the subspace $X_N\left(\Wast;\widetilde{H}^1_\alpha(D^\Lambda_H)\right)$ has the representation
\begin{equation}\label{eq:approx_finit}
w_B^N(\alpha,x)=C_\Lambda\sum_{j\in\Z_N} u_T\left(x+\left(\begin{matrix}
\Lambda j\\0
\end{matrix}
\right)\right)e^{-\i\alpha\Lambda j}.
\end{equation}
The error of the approximation is estimated in the following theorem. 

\begin{theorem}\label{th:approx}
If $w_B\in H_0^r(\Wast;H_\alpha^s(D^\Lambda_H))$, then for any $r'<r$,
\begin{equation}
\|w_B-w_B^N\|_{H_0^{r'}(\Wast;H_\alpha^s(D^\Lambda_H))}\leq N^{r'-r}\|w_B\|_{H_0^r(\Wast;H_\alpha^s(D^\Lambda_H))}.
\end{equation}
\end{theorem}

\begin{proof}
From the definition of $w_B$,
\begin{equation*}
w_B-w_B^N=C_\Lambda\sum_{|j|>N}u_T\left(x+\left(\begin{matrix}
\Lambda j\\0
\end{matrix}
\right)\right)e^{-\i\alpha\Lambda j}.
\end{equation*}
As $w_B\in H_0^r(\Wast;H_\alpha^s(D^\Lambda_H))$, from Remark \ref{rem:Four}, 
\begin{equation*}
\|w_B\|_{H_0^r(\Wast;H_\alpha^s(D^\Lambda_H))}^2=\sum_{\ell\in\Z}(1+|\ell|^2)^r\left\|u_T(\cdot+(\Lambda j,0)^\top)\right\|^2_{H^s_\alpha(D^\Lambda_H)}<\infty.
\end{equation*}
We can obtain the norm of $w_B-w_B^N$ in the same way, i.e.,
\begin{equation*}
\begin{aligned}
\|w_B-w_B^N\|_{H_0^{r'}(\Wast;H_\alpha^s(D^\Lambda_H))}^2&=\sum_{\ell\in\Z\setminus\Z_N}(1+|\ell|^2)^{r'}\left\|u_T(\cdot+(\Lambda j,0)^\top)\right\|^2_{H^s_\alpha(D^\Lambda_H)}\\
&\leq (1+(N/2)^2)^{r'-r}\sum_{|\ell|>N}(1+|\ell|^2)^r\left\|u_T(\cdot+(\Lambda j,0)^\top)\right\|^2_{H^s_\alpha(D^\Lambda_H)}\\
&\leq (N/2)^{2(r'-r)}\|w_B\|^2_{H_0^r(\Wast;H_\alpha^s(D^\Lambda_H))}.
\end{aligned}
\end{equation*}
So $\|w_B-w_B^N\|_{H_0^{r'}(\Wast;H_\alpha^s(D^\Lambda_H))}\leq (N/2)^{r'-r}\|w_B\|_{H_0^r(\Wast;H_\alpha^s(D^\Lambda_H))}$. The proof is finished.

\end{proof}

Following the procedure in the first section, we can redefine the indicator function $\mathcal{X}_N$ in the two dimensional space by
\begin{equation*}
\mathcal{X}_N=\begin{cases}
1,\quad\text{ in }\cup_{j\in\Z_N}\left[\W+\Lambda j\right]\times\R;\\
0,\quad\text{ otherwise},
\end{cases}
\end{equation*}
and define the truncated functions of $c$ and $A$ by
\begin{equation*}
c_N(x)=c(x)\mathcal{X}_N(x);\quad A_N(x)=A(x)\mathcal{X}_N(x).
\end{equation*}
Then the sesquilinear form $b(w_B,v_B)$ has the representation
\begin{equation*}
\begin{aligned}
b(w_B^N,v_B)&=\int_{D_H}\Big[ A\grad( \J_{D_H}^{-1} w_B^N)\cdot\grad\overline{\left(\J_{D_H}^{-1} \overline{v_B}\right)}-k^2 c(\J_{D_H}^{-1}w_B^N)\cdot\overline{\left(\J_{D_H}^{-1} \overline{v_B}\right)}\Big]\d x\\
&=\int_{D_H}\Big[ A\grad( \J_{D_H}^{-1} w_B)\cdot\grad\overline{\left(\J_{D_H}^{-1} \overline{v_B}\right)}-k^2 c(\J_{D_H}^{-1}w_B)\cdot\overline{\left(\J_{D_H}^{-1} \overline{v_B}\right)}\Big]\mathcal{X}_N(x)\d x\\
&=\int_{D_H}\Big[ A_N\grad( \J_{D_H}^{-1} w_B)\cdot\grad\overline{\left(\J_{D_H}^{-1} \overline{v_B}\right)}-k^2 c_N(\J_{D_H}^{-1}w_B)\cdot\overline{\left(\J_{D_H}^{-1} \overline{v_B}\right)}\Big]\d x.
\end{aligned}
\end{equation*}
Define 
\begin{equation*}
b_N(w,v)=\int_{D_H}\Big[ A_N\grad( \J_{D_H}^{-1} w)\cdot\grad\overline{\left(\J_{D_H}^{-1} \overline{v}\right)}-k^2 c_N(\J_{D_H}^{-1}w)\cdot\overline{\left(\J_{D_H}^{-1} \overline{v}\right)}\Big]\d x,
\end{equation*}
then
\begin{equation*}
b(w_B^N,v_B)=b_N(w_B,v_B).
\end{equation*}

We can similarly approximate $v_B$ by $v_B^N\in X_N\left(\Wast;\widetilde{H}^1_\alpha(D^\Lambda_H)\right)$, then 
\begin{equation}\label{eq:ort}
b(w_B^N,v_B^N)=b(w_B^N,v_B)=b(w_B,v_B^N)=b_N(w_B,v_B)=b_N(w_B^N,v_B^N).
\end{equation}

\subsection{Truncated Bloch transformed fields}

Recall the variational form \eqref{eq:var_Bloch}:
\begin{equation*}
\int_\Wast a_\alpha(w_B(\alpha,\cdot),v_B(\alpha,\cdot))\d\alpha+b(w_B,v_B)=\int_\Wast\int_{\Gamma^\Lambda_H}F(\alpha,x)\overline{v_B(\alpha,x)}\d s(x)\d\alpha
\end{equation*}
for any $v_B\in H_0^{-r}(\Wast;\widetilde{H}^1_\alpha(D^\Lambda_H))$. Replace $v_B$ by $v_B^N$, from the orthogonality and \ref{eq:ort},
\begin{equation*}
\int_\Wast a_\alpha(w_B^N(\alpha,\cdot),v_B^N(\alpha,\cdot))\d\alpha+b_N(w_B^N,v_B^N)=\int_\Wast\int_{\Gamma^\Lambda_H}F(\alpha,x)\overline{v_B^N(\alpha,x)}\d s(x)\d\alpha.
\end{equation*}
Replace $v_B$ by $v_B-v_B^N$ in the variational form, use the orthogonality again,
 \begin{equation*}
\int_\Wast a_\alpha(w_B^N(\alpha,\cdot),\left(v_B-v_B^N\right)(\alpha,\cdot))\d\alpha=0;\quad b_N(w_B^N,v_B-v_B^N)=0.
\end{equation*}
Thus $w_B^N$ satisfies the variational problem
\begin{equation}\label{eq:var_trunc}
\int_\Wast a_\alpha(w_B^N(\alpha,\cdot),v_B(\alpha,\cdot))\d\alpha+b_N(w_B^N,v_B)=\int_\Wast\int_{\Gamma^\Lambda_H}F(\alpha,x)\overline{v_B(\alpha,x)}\d s(x)\d\alpha
\end{equation}
for any $v_B\in H_0^{-r}(\Wast;\widetilde{H}_\alpha^1(D^\Lambda_H))$.

Although the cut-offed functions $A_N$ and $c_N$ maybe no longer continuous, the well-posedness of the variational form \eqref{eq:var_trunc} still holds in weighted Sobolev spaces, as is shown in the next theorem.

\begin{theorem}\label{th:solv_trunc}
When $N$ is large enough and $|r|<1$, the variational problem \eqref{eq:var_trunc} is uniquely solvable in $H_0^{r'}(\Wast;\widetilde{H}_\alpha^1(D^\Lambda_H))$ for any $-1<r'<r$.
\end{theorem}

\begin{proof}As the variational problem \eqref{eq:var_trunc} is a perturbation of \eqref{eq:var_Bloch}, we only need to consider difference between the sesquilinear form $b(\cdot,\cdot)$ and $b_N(\cdot,\cdot)$. For any $w\in H_0^r(\Wast;\widetilde{H}^1_\alpha(D^\Lambda_H))$, the approximation in the finite dimensional subspace $X_N(\Wast;\widetilde{H}^1_\alpha(D^\Lambda_H))$ is defined by \eqref{eq:approx_finite}. From Theorem \ref{th:approx}, for any $r'<r$, the error between $w$ and its approximation $w_N$ in $H_0^{r'}(\Wast;\widetilde{H}_\alpha^1(D^\Lambda_H))$ is  bounded by
\begin{equation*}
\|w-w_N\|_{H_0^{r'}(\Wast;\widetilde{H}_\alpha^1(D^\Lambda_H))}\leq (N/2)^{r'-r}\|w\|_{H_0^r(\Wast;\widetilde{H}_\alpha^1(D^\Lambda_H))}.
\end{equation*}
Then from \eqref{eq:ort} and the boundedness of $b(\cdot,\cdot)$,
\begin{equation*}
\begin{aligned}
\left|b(w,v)-b_N(w,v)\right|&=\left|b(w-w_N,v)\right|\\
&\leq C \|w-w_N\|_{H_0^{r'}(\Wast;\widetilde{H}_\alpha^1(D^\Lambda_H))}\|v\|_{H_0^{-r'}(\Wast;\widetilde{H}_\alpha^1(D^\Lambda_H))}\\
&\leq C(N/2)^{r'-r}\|w\|_{H_0^r(\Wast;\widetilde{H}_\alpha^1(D^\Lambda_H))} \|v\|_{H_0^{-r'}(\Wast;\widetilde{H}_\alpha^1(D^\Lambda_H))}.
\end{aligned}
\end{equation*}
Thus when $N\rightarrow+\infty$, 
\begin{equation*}
b_N(w,v)\rightarrow b(w,v).
\end{equation*}
Thus when $N$ is large enough, \eqref{eq:var_trunc} is a small perturbation of \eqref{eq:var_Bloch} in $H_0^{r'}(\Wast;\widetilde{H}_\alpha^1(D^\Lambda_H))\times H_0^{-r'}(\Wast;\widetilde{H}_\alpha^1(D^\Lambda_H))$. From the well-posedness of \eqref{eq:var_Bloch}, the variational form \eqref{eq:var_trunc} is uniquely solvable in $H_0^{r'}(\Wast;\widetilde{H}_\alpha^1(D^\Lambda_H))$. The proof is finished.

\end{proof}

\section{The finite element method}

In this section, we discuss a Garlekin discretization for the scattering problems from rough surfaces. As was shown in the last section, the field $w_B(\alpha,\cdot)$ could be approximated by finite Fourier series $w_B^N(\alpha,x)$ defined by \eqref{eq:approx_finit}, which is exactly the solution of the truncated problem \eqref{eq:var_trunc}. Let the uniformly distributed grid points in $\Wast$ defined by
\begin{equation*}
\alpha_N^{(1)}=-\frac{\pi}{\Lambda}+\frac{\pi}{N\Lambda},\quad
\alpha_N^{(j)}=\alpha_N^{(j-1)}+\frac{2\pi}{N\Lambda}\in\Wast,\,j=2,\dots,N.
\end{equation*}
Then define the piecewise basic functions  $\left\{\psi_N^{(j)}\right\}_{j=1}^N$ such that for any $j=1,\dots,N$, $\psi_N^{(j)}$ equals to $1$ in the $j$-th interval $\left(\alpha_N^{j}-\pi/(N\Lambda),\alpha_N^{(j)}+\pi/(N\Lambda)\right]$ and equals to $0$ otherwise. Assume that $\mathcal{M}_h$ is a family of regular and quasi-uniform meshes (see \cite{Brenn1994}) for the periodic cell $D^\Lambda_H$,  where $0<h<h_0$ and $h_0$ is a small enough positive number. To obtain the periodic basic functions, it is required that the nodal points on the left and right boundaries have the same heights. By omitting the nodal points on the left boundary, let $\left\{\phi_M^{(\ell)}\right\}_{\ell=1}^M$ be the piecewise linear and globally continuous nodal functions equal to one at one point except for the lower boundary, and equal zero at other nodal points, then $\widetilde{V}_h:={\rm span}\left\{\phi_M^{(\ell)}\right\}_{\ell=1}^M$ is a subspace $\widetilde{H}_0^1(D^\Lambda_H)$. Then we can define the finite element space $\widetilde{X}_{N,h}$ by
\begin{equation*}
\widetilde{X}_{N,h}:=\left\{v_{N,h}(\alpha,x)=e^{-\i\alpha x_1}\sum_{j=1}^N\sum_{\ell=1}^M v_{N,h}^{(j,\ell)}\psi_N^{(j)}(\alpha)\phi_M^{(\ell)}(x):\, v_{N,h}^{(j,\ell)}\in\C\right\}.
\end{equation*}
It is easy to check that $\widetilde{X}_{N,h}\subset L^2(\Wast;\widetilde{H}^1_\alpha(D^\Lambda_H))$ following \cite{Lechl2017}. Moreover, from the definition of the basic functions, $\widetilde{X}_{N,h}\subset X_N\left(\Wast;\widetilde{H}_\alpha^1(D^\Lambda_H)\right)$. We will seek for a finite element solution $w_{N,h}\in \widetilde{X}_{N,h}$ to the truncated problem
\begin{equation}\label{eq:var_discrete}
\int_\Wast a_\alpha(w_{N,h},v_{N,h})\d\alpha+b_N(w_{N,h},v_{N,h})=\int_\Wast\int_{\Gamma^\Lambda_H}F(\alpha,x)\overline{v}_{N,h}\d s\d\alpha
\end{equation}
for any $v_{N,h}\in \widetilde{X}_{N,h}$. 

From the definition of $b_N(\cdot,\cdot)$, it could be written into the finite sum
\begin{equation*}
b_N(w,v)=\sum_{m\in\Z_N}b_N^{(m)}\left(( \J_{D_H}^{-1} w)\left(\begin{smallmatrix}
x_1+\Lambda m\\x_2
\end{smallmatrix}
\right),{\left(\J_{D_H}^{-1} \overline{v}\right)\left(\begin{smallmatrix}
x_1+\Lambda m\\x_2
\end{smallmatrix}
\right)}\right),
\end{equation*}
where
\begin{equation*}
b_N^{(m)}(w,v)=\int_{D_H^\Lambda}\left[ A_N^{(m)}(x)\grad w\cdot\grad\overline{v}-k^2 c_N^{(m)}(x)w\overline{v}\right]\d x.
\end{equation*}
In the definition, $A_N^{(m)}(x)=A\left(\begin{smallmatrix}x_1+\Lambda m\\x_2\end{smallmatrix}\right)$, $c_N^{(m)}(x)=c\left(\begin{smallmatrix}x_1+\Lambda m\\x_2\end{smallmatrix}\right)$. The inverse Bloch transform can be explicitly computed
\begin{equation}\label{eq:inv_Bloch}
\begin{aligned}
&\left(\J_{D_H}^{-1} w_{N,h}\right)\left(x+\left(\begin{matrix}
\Lambda m\\0
\end{matrix}\right)\right)=C_\Lambda\int_\Wast w_{N,h}(\alpha,x)e^{\i\alpha\cdot\Lambda m}\d\alpha\\
=&C_\Lambda\int_\Wast \left[e^{-\i\alpha x_1}\sum_{j=1}^N\sum_{\ell=1}^M w_{N,h}^{(j,\ell)}\psi_N^{(j)}(\alpha)\phi_M^{(\ell)}(x)\right]e^{\i\alpha\cdot\Lambda m}\d\alpha\\
=&C_\Lambda\sum_{j=1}^N g_N^{(j,m)}(x_1)\sum_{\ell=1}^M w_{N,h}^{(j,\ell)}\phi_M^{(\ell)}(x):=\J_{D_H,N,m}^{-1}\left(\left\{w_{N,h}^{(j,\ell)}\right\}_{j,\ell=1}^{N,M}\right),
\end{aligned}
\end{equation}
where 
\begin{equation*}
g_N^{(j,m)}(x_1)=\i e^{-\i\alpha_N^{(j)}(x_1-\Lambda m)}\frac{e^{-\i\pi(x_1-\Lambda m)/(N\Lambda)}-e^{\i\pi(x_1-\Lambda m)/(N\Lambda)}}{x_1-\Lambda m}\quad\text{ if }x_1\neq \Lambda m, 
\end{equation*}
and
\begin{equation*}
g_N^{(j,m)}(x_1)=\frac{2\pi}{N\Lambda}\quad\text{ if }x_1= \Lambda m. 
\end{equation*}
 For details see the next section for the numerical implementation.

 The well-posedness and convergence of the finite dimensional problem \eqref{eq:var_discrete}  could be obtained.
 
\begin{theorem}\label{th:err_approx}
Assume that $f\in H_r^{1/2}(\Gamma_H)$ for $r> 1/2$ and $\zeta\in C^{2,1}(\R)$. Then the linear system \eqref{eq:var_discrete} is uniquely solvable in $\widetilde{X}_{N,h}$ for any $F(\alpha,\cdot)=\left(\J_{\Gamma_H} f\right)(\alpha,\cdot)$ in $H_0^r(\Wast;H^{1/2}_\alpha(\Gamma^\Lambda_H))$, when $N\geq N_0$  and $0<h<h_0$, where $N_0$ is sufficiently large and $h_0>0$ is small enough. The solution $w_{N,h}\in \widetilde{X}_{N,h}$ satisfies the error estimate for any $1/2<r'<r$:
\begin{equation}
\|w_{N,h}-w_B^N\|_{L^2(\Wast;H^\ell(D^\Lambda_H))}\leq Ch^{1-\ell}\left(N^{-r'}+h\right)\|f\|_{H_{r'}^{1/2}(\Gamma_H)},\quad\ell=0,1.
\end{equation}
\end{theorem}

\begin{proof}
From Theorem 9 in \cite{Lechl2016a}, we only need to prove the solvability of the truncated problem \eqref{eq:var_trunc} when $r>1/2$ and $\zeta\in C^{2,1}(\R)$.  When $\zeta\in C^{2,1}(\R)$, from Theorem \ref{th:higher_reg}, $w_B\in H_0^r(\Wast;\widetilde{H}_\alpha^2(D^\Lambda_H))$. Following the proof in Theorem  Theorem \ref{th:solv_trunc}, as $b_N(\cdot,\cdot)$ is also a small perturbation of  $b(\cdot,\cdot)$ defined in $H_0^r(\Wast;\widetilde{H}_\alpha^2(D^\Lambda_H))\times H_0^{-r}(\Wast;\widetilde{H}_\alpha^2(D^\Lambda_H))$, we can also prove that $w_B^N\in H_0^{r'}(\Wast;\widetilde{H}_\alpha^2(D^\Lambda_H))$ for any $-1<r'<r$. Thus when $r>1/2$, we can find a $1/2<r'<r$, $w_B^N\in H_0^{r'}(\Wast;\widetilde{H}_\alpha^2(D^\Lambda_H))$. The rest of the proof is omitted for it is the same as the proof in \cite{Lechl2016a}.
\end{proof}

Thus the error estimation between the approximate $w_{N,h}$ and $w_B$ could be obtained in the following theorem.

\begin{theorem}\label{th:err}
Assume that $f,\,r,\,\zeta$ satisfy the conditions in Theorem \ref{th:err}. Then the solution $w_{N,h}\in\widetilde{X}_{N,h}$ satisfies the error estimate for any $1/2<r'<r$:
\begin{equation}
\|w_{N,h}-w_B\|_{L^2\left(\Wast;H^\ell(D^\Lambda_H)\right)}\leq C\left(N^{-r'}+h^{2-\ell}\right)\|f\|_{H_{r}^{1/2}(\Gamma_H)},\quad\ell=0,1.
\end{equation}
\end{theorem}

\begin{proof}
From Theorem \ref{th:approx}, the error between the original solution and the truncated one is bounded by
\begin{equation*}
\|w_B-w_B^N\|_{L^2\left(\Wast;H_\alpha^1(D^\Lambda_H)\right)}\leq N^{-r}\|w_B\|_{H_0^{r}\left(\Wast;H_\alpha^1(D^\Lambda_H)\right)}\leq C N^{-r}\|f\|_{H_r^{1/2}(\Gamma_H)},\quad \ell=0,1.
\end{equation*}
Together with the result in Theorem \ref{th:err_approx},
\begin{equation*}
\begin{aligned}
\|w_B-w_{N,h}\|_{L^2\left(\Wast;H_\alpha^\ell(D^\Lambda_H)\right)}&\leq \|w_B-w_B^{N}\|_{H_0^{r'}\left(\Wast;H_\alpha^\ell(D^\Lambda_H)\right)}+\|w_N^B-w_{N,h}\|_{H_0^{r'}\left(\Wast;H_\alpha^\ell(D^\Lambda_H)\right)}\\
&\leq  C N^{-r}\|f\|_{H_r^{1/2}(\Gamma_H)}+Ch^{1-\ell}\left(N^{-r'}+h\right)\|f\|_{H_{r'}^{1/2}(\Gamma_H)}\\
&\leq C(N^{-r'}+h^{2-\ell})\|f\|_{H_{r}^{1/2}(\Gamma_H)}.
\end{aligned}
\end{equation*}
The proof is finished.
\end{proof}

\section{Numerical implementation for rough surfaces}

In this section, we describe the numerical implementation for the variational problem \eqref{eq:var_discrete}. For convenience, the quasi-periodic fields are periodized and the scattered field is considered instead of the total field. 

Similar to \cite{Lechl2017}, define $w^s(\alpha,x):=w_B(\alpha,x)-(\J_{D_H}u^i)(\alpha,x)$ and then define $w_0(\alpha,x)=e^{\i\alpha x_1}w^s(\alpha,x)$, then $w_0\in L^2(\Wast;H_0^1(D^\Lambda_H))$ (this space means that each function is $\Lambda$-periodic in $x_1$-direction for any fixed $\alpha$). Thus for any fixed $\alpha$, $w_0(\alpha,\cdot)$ is a periodic function in $D^\Lambda_H$. Let $v_0=e^{\i\alpha x_1}v_B(\alpha,x)$, then the variational formulation for $w_0$ is
\begin{equation*}
\int_\Wast a'_\alpha(w_0(\alpha,\cdot),v_0(\alpha,\cdot))\d\alpha+b'(w_0,v_0)=0
\end{equation*}
with the the boundary condition
\begin{equation}\label{eq:boundary}
\int_\Wast\int_{\Gamma_{h_0}^\Lambda}\left[w_0(\alpha,\cdot)-e^{\i\alpha x_1}(\J_{D_H}u^i)(\alpha,\cdot)\right]\overline{t_0}(\alpha,x)\d s\d\alpha=0
\end{equation}
for all $t\in L^2(\Wast;H_0^{-1/2}(\Gamma_0^\Lambda))$. The sesquilinear forms $a'(\cdot,\cdot)$ is defined in $H_0^1(D^\Lambda_H)\times H_0^1(D^\Lambda_H)$ by
\begin{equation*}
a'(w_0,v_0)=\int_{D^\Lambda_H}\left[(\grad_x+\i\alpha{\bm e}_1 )w_0\cdot(\grad_x-\i\alpha{\bm e}_1)\overline{v_0}-k^2w_0\overline{v_0}\right]\d x-\int_{\Gamma^\Lambda_H}\widetilde{T}^+_\alpha\left[w_0\big|_{\Gamma^\Lambda_H}\right]\d s
\end{equation*}
and $b'(\cdot,\cdot)$ is defined in $L^2(\Wast;H_0^1(D^\Lambda_H))\times L^2(\Wast;H_0^1(D^\Lambda_H))$ by
\begin{equation*}
b'(w_0,v_0)=b(e^{-\i\alpha \cdot}w_0,e^{-\i\alpha\cdot}v_0),
\end{equation*}
where ${\bm e}_1=(1,0)^\top$ and $\widetilde{T}^+_\alpha$ is the modified Dirichlet-to-Neumann map defined on the periodic functions on $\Gamma^\Lambda_H$:
\begin{equation*}
\widetilde{T}^+_\alpha(\phi)=\i\sum_{j\in\Z}\sqrt{k^2-|\Lambda^*j-\alpha|^2}\widehat{\phi}(j)e^{\i\Lambda j x_1}\quad\text{ for } \phi(x)=\widehat{\phi}(j)e^{\i\Lambda j x_1}.
\end{equation*} 
Thus the finite Fourier series approximation \eqref{eq:approx_finit}, denoted by $w_0^N$, satisfies the variational formulation
\begin{equation}\label{eq:var_per}
\int_\Wast a'_\alpha(w_0^N(\alpha,\cdot),v_0(\alpha,\cdot))\d\alpha+b'_N(w_0^N,v_0)=0
\end{equation}
together with the same boundary condition \eqref{eq:boundary}, where $b'_N(\cdot,\cdot)$ is defined by
\begin{equation*}
b'_N(w_0,v_0):=b_N(e^{-\i\alpha\cdot}w_0,e^{-\i\alpha\cdot}v_0).
\end{equation*}

Now we can discretize the variational formulation \eqref{eq:var_per}. Recall that the  nodal functions $\left\{\phi_M^{(\ell)}\right\}_{\ell=1}^M$ that are periodic and vanishes on the boundary $\Gamma_{h_0}^\Lambda$, we have to introduce the basic functions on the nodal points on $\Gamma_{h_0}^\Lambda$. Suppose $M'$ is an integer larger than $M$, $x_{M'}^{(\ell)}$ are nodal points on the mesh $\mathcal{M}_h$, where $x_{M'}^{(\ell)},\,\ell=1,\dots,M$ does not lie on $\Gamma^\Lambda_{h_0}$ and $x_{M'}^{(\ell)},\,\ell=M+1,\dots,M'$ lies on $\Gamma^\Lambda_{h_0}$.  Let  $\left\{\psi_{M'}^{(\ell)}\right\}_{\ell=1}^{M'}$  be nodal functions that equals to one at one nodal point $x_{M'}^{(\ell)}$ and zero otherwise, assume that $\psi_{M'}^{(\ell)}=\psi_M^{(\ell)}$ for any $\ell=1,\dots,M$, then define 
\begin{equation*}
V_h:={\rm span}\left\{\phi_{M'}^{(\ell)}\right\}_{\ell=1}^{M'}\subset H_0^1(D^\Lambda_H).
\end{equation*}
Recall the basis functions $\left\{\psi_N^{(j)}\right\}_{j=1}^N$ in $\alpha\in\Wast$, we can define the new finite element space 
\begin{equation*}
X_{N,h}=\left\{v_{N,h}(\alpha,x)=\sum_{j=1}^N\sum_{\ell=1}^{M'}v_{N,h}^{(j,\ell)}\psi_N^{(j)}(\alpha)\phi_{M'}^{(\ell)}(x):\,v_{N,h}^{(j,\ell)}\in\C\right\}\subset L^2(\Wast;H_0^1(D^\Lambda_H)).
\end{equation*}
We can define the subspaces
\begin{equation*}
Y^{(j)}_{N,h}=\left\{v_{N,h}(\alpha,x)=\sum_{\ell=1}^{M'}v_{N,h}^{(j,\ell)}\psi_N^{(j)}(\alpha)\phi_{M'}^{(\ell)}(x):\,v_{N,h}^{(j,\ell)}\in\C\right\},\quad j=1,\dots,N.
\end{equation*}
and also the subspaces that vanish on $\Gamma^\Lambda_0$
\begin{equation*}
\widetilde{Y}^{(j)}_{N,h}=\left\{v_{N,h}(\alpha,x)=\sum_{\ell=1}^{M}v_{N,h}^{(j,\ell)}\psi_N^{(j)}(\alpha)\phi_{M'}^{(\ell)}(x):\,v_{N,h}^{(j,\ell)}\in\C\right\},\quad j=1,\dots,N.
\end{equation*}
Let
\begin{equation*}
Y_{N,h}^0=Y_{N,h}^{(1)}\oplus\cdots\oplus Y_{N,h}^{(N)},\quad \widetilde{Y}_{N,h}^0=\widetilde{Y}_{N,h}^{(1)}\oplus\cdots\oplus \widetilde{Y}_{N,h}^{(N)}.
\end{equation*}
Then for any $w_0\in X_{N,h}$, there is a unique vector $\left(w_0^{(j)}\right)_{j=1}^N\in Y_{N,h}^0$ such that 
\begin{equation*}
w_0(\alpha,x)=\sum_{j=1}^N w_0^{(j)}(x). 
\end{equation*}
Let $w_0^{(j)}=\sum_{\ell=1}^{M'}w_{N,h}^{(j,\ell)}\phi_{M'}^{(\ell)}(x)\psi_N^{(j)}(\alpha)$. The boundary term of \eqref{eq:var_per} could  be approximated in the similar way of \cite{Lechl2017}:
\begin{equation*}
w_0^{(j)}(x_{M'}^{(\ell)})=\int_{\Wast}e^{\i\alpha \left(x_{M'}^{\ell}\right)_1}(\J_{D_H}u^i)\left(\alpha,x_{M'}^{\ell}\right)\psi_N^{(j)}(\alpha)\d\alpha
\end{equation*}
for $j=1,\dots,N$ and $\ell=M+1,\dots,M'$. Thus
\begin{equation*}
w_{N,h}^{(j,\ell)}= \exp\left(\i\alpha_j \left(x_{M'}^{\ell}\right)_1\right)(\J_{D_H}u^i)\left(\alpha_j,x_{M'}^{\ell}\right)(:=c_{j,\ell}).
\end{equation*}
 For  $m=1,\dots,M$ and $n=1,\dots,N$, let $v_{m,n}=\phi_{M'}^{(m)}(x)\psi_N^{(n)}(\alpha)$. Then
\begin{equation*}
\begin{aligned}
\int_\Wast a'_\alpha(w_0,v_{m,n})\d\alpha&=\sum_{j=1}^N\sum_{\ell=1}^{M'}w_{N,h}^{(j,\ell)}\, \int_\Wast a'_\alpha\left(\phi_{M'}^{(\ell)},\phi_{M'}^{(m)}\right)\psi_N^{(j)}(\alpha)\psi_N^{(n)}(\alpha)\d\alpha\\
&:=\sum_{j=1}^N\sum_{\ell=1}^{M'}a_{j,\ell,n,m} w_{N,h}^{(j,\ell)}\delta_{j,n},
\end{aligned}
\end{equation*}
where $\delta_{j,n}$ equals to $1$ when $j=n$ and equals to $0$ otherwise, the coefficient is defined by
\begin{equation*}
a_{j,\ell,n,m}=\int_{\alpha_N^{(n)}-\frac{\pi}{N\Lambda}}^{\alpha_N^{(n)}+\frac{\pi}{N\Lambda}} a'_{\alpha}\left(\phi_{M'}^{(\ell)},\phi_{M'}^{(m)}\right)\d\alpha.
\end{equation*}
Then consider the discretization of the term $b'(\cdot,\cdot)$. Recall the inverse Bloch transform in \eqref{eq:inv_Bloch},
\begin{equation*}
\left(\J_{D_H}^{-1}w_0 e^{-\i\alpha(\cdot)_1}\right)\left(x+\left(\begin{matrix}
\Lambda \ell'\\0
\end{matrix} \right)\right)=\J_{D_H,N,\ell'}^{-1}\left(\left\{w_{N,h}^{j,\ell}e^{-\i\alpha (\cdot)_1}\right\}_{j,\ell=1}^{N,M'}\right),
\end{equation*}
where $(\cdot)_1$ is the first component of the variable. 
Thus the discrete form of $b'_N(\cdot,\cdot)$ 
\begin{equation*}
\begin{aligned}
&b'_N(w_0,v_{m,n})=\sum_{\ell'\in\Z_N}b_N^{(\ell')}\left(\left(\J_{D_H}^{-1}w_0 e^{-\i\alpha(\cdot)_1}\right)\left(x+\left(\begin{smallmatrix}
\Lambda\ell'\\0
\end{smallmatrix}
\right)\right),\left(\J_{D_H}^{-1}\phi_{M'}^{(m)}\psi_N^{(n)}e^{-\i\alpha(\cdot)_1}\right)\left(x+\left(\begin{smallmatrix}
\Lambda\ell'\\0
\end{smallmatrix}
\right)\right)\right)\\
&=\sum_{\ell'\in\Z_N}\sum_{j=1}^N\sum_{\ell=1}^{M'} b_N^{(\ell')}\left(\J_{D_H,N,\ell'}^{-1}\left(\left\{w_{N,h}^{j,\ell}e^{-\i\alpha (\cdot)_1}\right\}_{j,\ell=1}^{N,M'}\right),\J_{D_H,N,\ell'}^{-1}\left(\left\{\delta_{j,\ell}^{m,n}e^{-\i\alpha (\cdot)_1}\right\}_{j,\ell=1}^{N,M'}\right)\right)\\
&:=\sum_{\ell'\in\Z_N}\sum_{j=1}^N\sum_{\ell=1}^{M'}b_{j,
\ell,n,m}^{\ell'} w_{N,h}^{(j,\ell)}
\end{aligned}
\end{equation*}
where $\delta_{j,\ell}^{m,n}$ equals to $1$ if and only if $j=m$ and $\ell=n$, 
\begin{equation*}
b_{j,\ell,n,m}^{\ell'}=b_N^{(\ell')}\left(\J_{D_H,N,\ell'}^{-1}\left(\left\{\delta_{j_1,\ell_1}^{j,\ell}e^{-\i\alpha (\cdot)_1}\right\}_{j_1,\ell_1=1}^{N,M'}\right),\J_{D_H,N,\ell'}^{-1}\left(\left\{\delta_{j,\ell}^{m,n}e^{-\i\alpha (\cdot)_1}\right\}_{j,\ell=1}^{N,M'}\right)\right).
\end{equation*}
Thus $w_0^{(j)}$ satisfies the linear system
\begin{eqnarray}\label{eqn1}
 \sum_{\ell=1}^{M'}a_{j,\ell,n,m}\delta_{j,n} w_{N,h}^{(j,\ell)}+\sum_{\ell'\in\Z_N}\sum_{\ell=1}^{M'}b_{j,\ell,n,m}^{\ell'} w_{N,h}^{(j,\ell)}&=&0;\\\label{eqn2}
 w_{N,h}^{j,\ell}&=&c_{j,\ell}.
\end{eqnarray}
Let $W_j=\left(w_{N,h}^{(j,1)},\dots,w_{N,h}^{(j,M')}\right)^\top$ and $F_j=\left(F_{(j,1)},\dots,F_{(j,M')}\right)^\top$ where $F_{(j,\ell)}=0$ for $\ell=1,\dots,M$ and $F_{(j,\ell)}=c_{j,\ell}$ when $\ell=M+1,\dots,M'$. Then \eqref{eqn1}-\eqref{eqn2} is equivalent to the following linear system
\begin{equation}\label{eq:system}
(A+B){\bm W}={\bm F}
\end{equation}
where
\begin{equation*}
A=\left(\begin{matrix}
A_1 & 0 & \cdots & 0\\
0 & A_2 & \cdots & 0\\
\vdots & \vdots & \vdots & \vdots \\
0 & 0 & \cdots & A_N
\end{matrix}
\right),\,A=\left(\begin{matrix}
B_{11} & B_{12} & \cdots & B_{1N}\\
B_{21} & B_{22} & \cdots & B_{2N}\\
\vdots & \vdots & \vdots & \vdots \\
B_{N1} & B_{N2} & \cdots & B_{NN}
\end{matrix}
\right),\, {\bm W}=\left(\begin{matrix}
W_1 \\ W_2 \\ \vdots \\ W_N
\end{matrix}
\right), \, {\bm F}=\left(\begin{matrix}
F_1 \\ F_2 \\ \vdots \\ F_N
\end{matrix}
\right)
\end{equation*}
where $A_{j,n}(m,\ell)=\left(a_{j,\ell,n,m}\right)_{\ell,m}$ for $1\leq m\leq M'$ and $1\leq \ell\leq M$, $A_j(m,\ell)=\delta_{m,\ell}$ otherwise, $B_{j,n}(\ell,m)=\sum_{\ell'\in\Z_N}b_{j,\ell,m,n}^{\ell'}$ for $1\leq m\leq M$ and $1\leq \ell\leq M$.

To solve the linear system \eqref{eq:system} of size $NM'\times NM'$, the iterative method is introduced for a fast convergence rate. The GMRES iteration scheme with a pre-conditioner is utilized, and is described in the following steps:
\begin{enumerate}
\item For each matrix $A_j$, let the incomplete LU decomposition  be $(L_j, U_j)$ for $j=1,\dots,N$. Then let the lower triangular matrix $L={\rm diag}(L_1,\dots,L_N)$ and the upper triangular matrix $U={\rm diag}(U_1,\dots,U_N)$.
\item Solve the linear system \eqref{eq:system} by GMRES with $(L,U)$ as the pre-conditioner.
\end{enumerate}

\section{Numerical examples}

In this section, we show four numerical examples of the rough surface scattering problems. We choose two rough surfaces above the straight line $\Gamma_{h_0}$, defined by the functions
\begin{equation*}
\zeta_1=1.1;\quad \zeta_2=1+0.1\sin(2.4 t),
\end{equation*}
then the surfaces are defined by
\begin{equation*}
\Gamma_j:=\left\{(x_1,\zeta_j(x_1)):\,x_1\in\R\right\},\text{ where }j=1,2.
\end{equation*}

Let the incident fields be points sources located at two different points, i.e., 
\begin{equation*}
P_1=(0.5,0.4);\quad P_2=(\pi,0.2).
\end{equation*}
If $y=(y_1,y_2)^\top$ is the location of the point source, then  the half space Green's function is defined by
\begin{equation*}
G(x,y)=\frac{\i}{4}\left[H_0^{(1)}(k|x-y|)-H_0^{(1)}(k|x-y'|)\right],\quad y=(y_1,-y_2)^\top.
\end{equation*}
From \cite{Lechl2016a}, for any fixed $y$, $G(\cdot,y)\in H_r^1(D_H)$ for any $r<1$. As the field $G(\cdot,y)$ is propagating upwards, the scattered field 
\begin{equation*}
u^s=G(\cdot,y) \quad\text{ on }\Gamma
\end{equation*}
 is exactly the function $G(\cdot,y)$. 

In the numerical examples, the following parameters are chosen:
\begin{equation*}
\Lambda=2\pi,\,\Lambda^*=1,\, H=3,\,H_0=2.95,\,h_0=1.
\end{equation*}
The numerical scheme is carried out for the mesh size $h$  chosen as $0.16,\,0.08,\,0.04,\,0.02$ and the parameter $N$  taken as $10,20,40,80$. Then the following four examples are considered for different $h$ and $N$, and the relative $L^2$-errors on $\Gamma_H$, defined by
\begin{equation*}
err=\frac{\|u_{N,h}-u\|_{L^2(\Gamma^\Lambda_H)}}{\|u\|_{L^2(\Gamma^\Lambda_H)}}
\end{equation*}
are listed in Table \ref{surf1k1}-\ref{surf2k2}.

The examples are chosen by different wave numbers, locations of point sources and rough surfaces:\\

\noindent{\bf Example 1} The wave number $k=1$, the point source is located at $P_1$, the rough surface is $\Gamma_1$, the relative errors are listed in Table \ref{surf1k1}.\\

\noindent{\bf Example 2} The wave number $k=6$, the point source is located at $P_1$, the rough surface is $\Gamma_1$, the relative errors are listed in Table \ref{surf1k2}.\\

\noindent{\bf Example 3} The wave number $k=1$, the point source is located at $P_2$, the rough surface is $\Gamma_2$, the relative errors are listed in Table \ref{surf2k1}.\\

\noindent{\bf Example 4} The wave number $k=6$, the point source is located at $P_2$, the rough surface is $\Gamma_2$, the relative errors are listed in Table \ref{surf2k2}.\\

\begin{table}[htb]
\centering
\caption{Relative $L^2$-errors for Example 1.}\label{surf1k1}
\begin{tabular}
{|p{1.8cm}<{\centering}||p{2cm}<{\centering}|p{2cm}<{\centering}
 |p{2cm}<{\centering}|p{2cm}<{\centering}|p{2cm}<{\centering}|}
\hline
  & $h=0.16$ & $h=0.08$ & $h=0.04$ & $h=0.02$\\
\hline
\hline
$N=10$&$4.6$E$-02$&$4.6$E$-02$&$4.6$E$-02$&$4.6$E$-02$\\
\hline
$N=20$&$1.7$E$-02$&$1.6$E$-02$&$1.6$E$-02$&$1.6$E$-02$\\
\hline
$N=40$&$6.9$E$-03$&$5.9$E$-03$&$5.7$E$-03$&$5.7$E$-03$\\
\hline
$N=80$&$4.5$E$-03$&$2.3$E$-03$&$2.1$E$-03$&$2.0$E$-03$\\
\hline
\end{tabular}
\end{table}

\begin{table}[htb]
\centering
\caption{Relative $L^2$-errors for Example 2.}\label{surf1k2}
\begin{tabular}
{|p{1.8cm}<{\centering}||p{2cm}<{\centering}|p{2cm}<{\centering}
 |p{2cm}<{\centering}|p{2cm}<{\centering}|p{2cm}<{\centering}|}
\hline
  & $h=0.16$ & $h=0.08$ & $h=0.04$ & $h=0.02$\\
\hline
\hline
$N=10$&$3.3$E$-01$&$1.0$E$-01$&$8.9$E$-02$&$1.1$E$-01$\\
\hline
$N=20$&$3.2$E$-01$&$9.1$E$-02$&$3.8$E$-02$&$3.8$E$-02$\\
\hline
$N=40$&$3.2$E$-01$&$8.9$E$-02$&$2.5$E$-02$&$1.5$E$-02$\\
\hline
$N=80$&$3.2$E$-01$&$8.9$E$-02$&$2.3$E$-02$&$7.6$E$-03$\\
\hline
\end{tabular}
\end{table}

\begin{table}[htb]
\centering
\caption{Relative $L^2$-errors for Example 3.}\label{surf2k1}
\begin{tabular}
{|p{1.8cm}<{\centering}||p{2cm}<{\centering}|p{2cm}<{\centering}
 |p{2cm}<{\centering}|p{2cm}<{\centering}|p{2cm}<{\centering}|}
\hline
  & $h=0.16$ & $h=0.08$ & $h=0.04$ & $h=0.02$\\
\hline
\hline
$N=10$&$5.9$E$-02$&$5.9$E$-02$&$5.8$E$-02$&$5.8$E$-02$\\
\hline
$N=20$&$2.2$E$-02$&$2.1$E$-02$&$2.1$E$-02$&$2.1$E$-02$\\
\hline
$N=40$&$8.7$E$-03$&$7.6$E$-03$&$7.3$E$-03$&$7.3$E$-03$\\
\hline
$N=80$&$4.5$E$-03$&$2.9$E$-03$&$2.6$E$-03$&$2.6$E$-03$\\
\hline
\end{tabular}
\end{table}

\begin{table}[htb]
\centering
\caption{Relative $L^2$-errors for Example 4.}\label{surf2k2}
\begin{tabular}
{|p{1.8cm}<{\centering}||p{2cm}<{\centering}|p{2cm}<{\centering}
 |p{2cm}<{\centering}|p{2cm}<{\centering}|p{2cm}<{\centering}|}
\hline
  & $h=0.16$ & $h=0.08$ & $h=0.04$ & $h=0.02$\\
\hline
\hline
$N=10$&$4.0$E$-01$&$1.1$E$-01$&$5.5$E$-02$&$6.3$E$-02$\\
\hline
$N=20$&$4.0$E$-01$&$1.1$E$-01$&$3.2$E$-02$&$2.3$E$-02$\\
\hline
$N=40$&$4.0$E$-01$&$1.1$E$-01$&$2.9$E$-02$&$1.0$E$-02$\\
\hline
$N=80$&$4.0$E$-01$&$1.1$E$-01$&$2.8$E$-02$&$7.3$E$-03$\\
\hline
\end{tabular}
\end{table}

From the numerical results in Table \ref{surf1k1}-\ref{surf2k2}, the relative error decreases when $N$ gets larger and $h$ gets less. For the wave number $k=1$, the error brought by $N$ is the dominant one, while the error comes from $h$ is relatively small. Then for small enough $h$'s, e.g., $h=0.02,\,0.04$ in Table \ref{surf1k1} and \ref{surf2k1}, the convergence rate with respect to $N$ could reach $O(N^{-1.5})$ (see Figure \ref{eg1}), which is even higher than expected in Theorem \ref{th:err}, i.e., $O(N^{-r'})$ for some $r'<1$. The examples for $k=6$ are the opposite, as the dominant error is caused by $h$. For large enough $N$'s, e.g., $N=80$ in Table \ref{surf1k2} and \ref{surf2k2}, the convergence rate with respect to $h$ could reach $O(h^{1.93})$, which is almost as high as expected (see Figure \ref{eg2}) in Theorem \ref{th:err}. Thus the numerical examples illustrate the convergence rate estimated in this paper.

\begin{figure}[H]
\centering
\includegraphics[width=7cm]{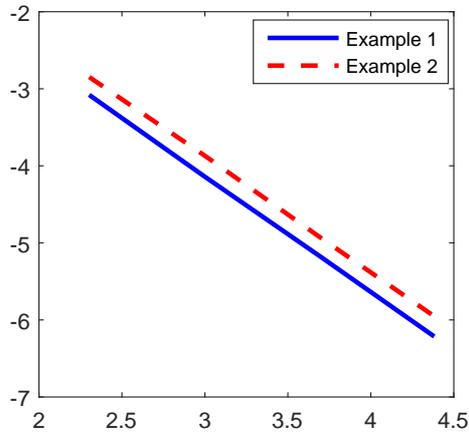}
\caption{The relative $L^2$-errors for Exampel 1 and 3 with $h=0.02$ plotted in logarithmic scale over $N$.}
\label{eg1}
\end{figure}

\begin{figure}[H]
\centering
\includegraphics[width=7cm]{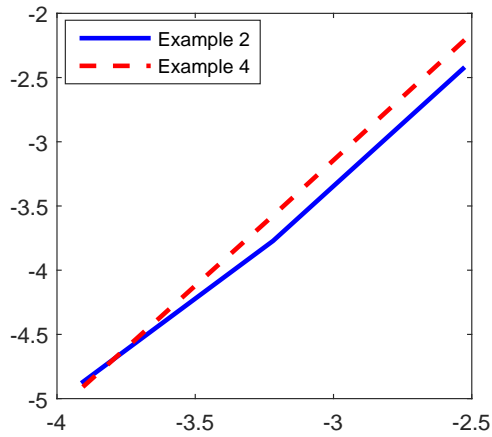}
\caption{The relative $L^2$-errors for Exampel 2 and 4 with $N=80$ plotted in logarithmic scale over $h$.}
\label{eg2}
\end{figure}

\section*{The Floquet-Bloch transform}

The main tool used in this paper is the Floquet-Bloch transform. In this section, we will recall the definition and some basic properties of the Bloch transform in periodic domains in $\R^2$ (for details see \cite{Lechl2016}).

Suppose $\Omega\subset\R^2$ is $\Lambda$-periodic in $x_1$ - direction, i.e., for any ${ x}=(x_1,x_2)^\top\in\Omega$, the translated point $(x_1+\Lambda j,x_2)\in\Omega,\,\forall{ j}\in\Z$. 
Define one periodic cell by $\Omega^\Lambda:=\Omega\cap\left[\W\times\R\right]$. For any $\phi\in C_0^\infty(\Omega)$, define the (partial)  Bloch transform in $\Omega$, i.e., $\J_{\Omega}$, of $\phi$ as
\begin{equation*}
\left(\J_\Omega\phi\right)({\alpha},{ x})=C_\Lambda\sum_{{ j}\in\Z}\phi\left({ x}+\left(\begin{matrix}
\Lambda { j}\\0
\end{matrix}\right)\right)e^{-\i{\alpha}\cdot\Lambda{ j}},\quad {\alpha}\in\R ,\,{ x}\in\Omega^\Lambda
\end{equation*}
where $C_\Lambda=\sqrt{\frac{\Lambda}{2\pi}}$.

\begin{remark}
The periodic domain $\Omega$ is not required to be bounded in $x_2$-direction.  
\end{remark}

We can also define the weighted Sobolev  space on the unbounded domain $\Omega$ by
\begin{equation*}
H_r^s(\Omega):=\left\{\phi\in \mathcal{D}'(\Omega):\,(1+|{ x}|^2)^{r/2}\phi({ x})\in H^s(\Omega)\right\}.
\end{equation*}
For any $\ell\in\N$, $s\in\R$, we can also define the following Hilbert space by
\begin{equation*}
H^\ell(\Wast;H^s(\Omega^\Lambda)):=\left\{\psi\in\mathcal{D}'(\Wast\times\Omega^\Lambda):\,\sum_{m=0}^\ell\int_\Wast\left\|\partial^m_{\alpha}\psi({\alpha},\cdot)\right\|\d{\alpha}<\infty\right\},
\end{equation*}
and extend to any $r\in\R$ by interpolation and duality arguments  similarly. The space $H_0^r(\Wast;H_\alpha^s(\Omega^\Lambda))$ could be defined in the same way. The following properties for the $d$-dimensional (partial) Bloch transform $\J_\Omega$ is also proved in \cite{Lechl2016}.

\begin{theorem}\label{th:Bloch_property}
The Bloch transform $\J_\Omega$ extends to an isomorphism between $H_r^s(\Omega)$ and $H_0^r(\Wast;H_\alpha^s(\Omega^\Lambda))$ for any $s,r\in\R$. Its inverse has the form of
\begin{equation*}
(\J^{-1}_\Omega\psi)\left({ x}+\left(\begin{matrix}
\Lambda { j}\\0
\end{matrix}\right)\right)=C_\Lambda\int_\Wast \psi({\alpha},{ x})e^{\i{\alpha}\cdot\Lambda{ j}}\d{\alpha},\quad x_1\in\Omega^\Lambda,\,{ j}\in\Z,
\end{equation*}
and the adjoint operator $\J^*_\Omega$ with respect to the scalar product in $L^2(\Wast;L^2(\Omega^\Lambda))$ equals to the inverse $\J^{-1}_\Omega$. Moreover, when $r=s=0$, the Bloch transform $\J_\Omega$ is an isometric isomorphism.
\end{theorem}

Another important property of the Bloch transform is the commutes with partial derivatives, see \cite{Lechl2016}. If $u\in H_r^n(\Omega)$ for some $n\in\N$, then for any ${ \gamma}=(\gamma_1,\gamma_2)\in\N^2$ with $|\gamma|=|\gamma_1|+|\gamma_2|\leq N$,
\begin{equation*}
\partial^{ \gamma}_{  x} \left(\J_\Omega u\right)({ \alpha},{  x})=\J_\Omega[\partial^{ \gamma} u]({ \alpha},{  x}).
\end{equation*}

\begin{remark}
The definition of the partial Bloch transform could also be extended to other periodic domains, for example, periodic hyper-surfaces. If $\Gamma$ is a $\Lambda$-periodic surface defined in $\R^2$, then we can define $\J_\Gamma$ in the same way, and obtain the same properties. In this paper, we will denote the Bloch transform $\J_X$ by the partial Bloch transform in the domain $X\subset\R^2$, which is periodic with respect to $x_1$-direction.
\end{remark}

\begin{remark}\label{rem:Four}
There is an alternative definition for the space $H_0^r(\Wast;X_{ \alpha})$, where $X_{ \alpha}$ is a family of Hilbert spaces that are ${ \alpha}$-quasi-periodic in $\widetilde{  x}$. Let 
\begin{equation*}
\phi_{\Lambda^*}^{({  j})}({ \alpha})=C_\Lambda e^{-\i{ \alpha}\cdot\Lambda {  j}},\,{  j}\in\Z
\end{equation*}
be a complete orthonormal system in $L^2(\Wast)$, then any function $\psi\in \mathcal{D}'(\Wast\times\Omega^\Lambda)$ has a Fourier series
\begin{equation*}
\psi({ \alpha},{  x})=C_\Lambda \sum_{{ \ell}\in\Z}\hat{\psi}_{\Lambda^*}({ \ell},{  x})e^{-\i{ \alpha}\cdot\Lambda{ \ell}},
\end{equation*}
where $\hat{\psi}_{\Lambda^*}({ \ell},{  x})=<\psi(\cdot,{  x}),\phi_{\Lambda^*}^{({ \ell})}>_{L^2(\Wast)}$. Then the squared norm of any $\psi\in H_0^r(\Wast;X_{ \alpha})$ equals to
\begin{equation*}
\|\psi\|^2_{H_0^r(\Wast;X_{ \alpha})}=\sum_{{ \ell}\in\Z}(1+|{ \ell}|^2)^r\left\|\hat{\psi}_{\Lambda^*}({ \ell},\cdot)\right\|^2_{X_{ \alpha}}.
\end{equation*}
\end{remark}

\bibliographystyle{alpha}
\bibliography{ip-biblio} 

\end{document}